\documentclass[reqno,12pt]{amsart}
\usepackage{amssymb,amsmath,amsthm,mathtools}
\usepackage{hyperref, url}
\usepackage[usenames]{color}
\usepackage{enumerate}
\def\Z{\mathbb{Z}}

\def\({\left(}
\def\){\right)}

\def\la{\lambda}
\let\temp\phi
\let\phi\varphi
\let\varphi\temp

\newtheorem{theorem}{Theorem}
\newtheorem{lemma}[theorem]{Lemma}

\theoremstyle{remark}
\newtheorem*{remark}{Remark}

\numberwithin{equation}{section}

\begin{document}


\title[ ]{Dissections of a ``strange" function}

\date{\today}
\dedicatory{In  memory of Felice Bateman, Paul Bateman, and Heini Halberstam}
\author{Scott Ahlgren}
\address{Department of Mathematics\\
University of Illinois\\
Urbana, IL 61801} 
\email{sahlgren@illinois.edu} 
 \author{Byungchan Kim}
\address{School of Liberal Arts \\ Seoul National University of Science and Technology \\ 172 Gongreung 2 dong, Nowongu, Seoul,139-743, Korea}
\email{bkim4@seoultech.ac.kr}

 
\begin{abstract}
The ``strange" function  of Kontsevich and Zagier  is defined by 
\[F(q):=\sum_{n=0}^\infty(1-q)(1-q^2)\dots(1-q^n).\] 
This series is defined only when $q$ is a root of unity, and provides an example 
of what Zagier has called a ``quantum modular form."
In their recent work on congruences for the Fishburn numbers $\xi(n)$ (whose generating function is $F(1-q)$), 
Andrews and Sellers recorded a speculation about the polynomials which appear in the dissections
of the partial sums of $F(q)$.  We prove that a more general form of their speculation is true.  The congruences  of Andrews-Sellers were generalized by Garvan 
in the case of prime modulus, and by Straub in the case of prime power modulus.  
As a corollary of our theorem, we reprove the known congruences for $\xi(n)$  modulo prime powers.
 \end{abstract}
 
\thanks{The first author was  supported by a grant from the Simons Foundation (\#208525 to Scott Ahlgren).
Byungchan Kim was supported by the Basic Science Research Program through the National Research Foundation of Korea (NRF) funded by the Ministry of Education (NRF 2013R1A1A2061326)}


\maketitle

\section{Introduction}
Define 
\[(q)_n:=(1-q)(1-q^2)\dots(1-q^n)\] 
and let the ``strange" function of Kontsevich and Zagier \cite{Zagier:2001} be defined by 
\begin{equation}
F(q):=\sum_{n=0}^\infty (q)_n.
\end{equation}
This is an example of what Zagier \cite{Zagier:2010} calls a ``quantum modular form."
The definition makes sense as a formal power series only when $q$ is a root of unity, in which case the sum terminates.
The series does not converge on any open set, but as shown by Zagier, the value of each of its derivatives is well defined at every root of unity.
In particular, if $\zeta$ is an $N$th root of unity, then we have \cite{Zagier:2001}
\begin{equation}\label{Fexpansion}
F(\zeta e^{-z})=\sum_{n=0}^\infty \frac{b_n(\zeta)z^n}{n!},
\end{equation}
where the coefficients are given by expressions of the form
\begin{equation}\label{bndefine}
b_n(\zeta)=\sum_{\substack{m\leq 6N \\ (m, 6)=1}} a(m, N)\zeta^{\frac{m^2-1}{24}}
\end{equation}
with certain integers $a(m, N)$.

We set
\[F(q; N):=\sum_{n=0}^N (q)_n.\]
If $\zeta_n$ is an $n$th root of unity and $N\geq n-1$ then we have
\begin{equation}\label{fnpartial}
F(\zeta_n)= F(\zeta_n; n-1)=F(\zeta_n; N).
\end{equation}

The Fishburn numbers $\xi(n)$ are defined by the generating function
\[\sum_{n=0}^\infty \xi(n)q^n=F(1-q)=1+q+2 q^2+5 q^3+15 q^4+53 q^5 +\dots.\]
This does make sense as a formal power series since $(1-q)_n=O(q^n)$.

We will consider the $t$-dissections of the polynomials $F(q; N)$ for positive integers $t$.  Given $N$ and $t$, we define polynomials
$A_{t} (N, i , q)\in \Z[q]$ via the formula
\begin{equation}\label{atdef}
F(q; N)  =\sum_{i=0}^{t-1} q^i A_{t} (N, i , q^t ).
\end{equation}
The pentagonal numbers are those integers of the form $(m^2-1)/24$ where $m$ is coprime to $6$.
We define $S(t)\subseteq\{0, 1, \dots, t-1\}$ as the set of  reductions of the collection of pentagonal numbers modulo $t$,
and we define 
\[\lambda(N, t):=\lfloor \tfrac{N+1}{t} \rfloor.\]

We will prove the following, which answers a question of Andrews and Sellers \cite{Andrews:2014} in the special case when $t$ is a prime number. 
\begin{theorem} \label{Atdiv}
Suppose that $t$ and $N$ are positive integers, and that 
 $i \not\in S(t)$.  Then  
 \[(q)_{\lambda(N, t)}\big | A_{t} (N, i, q).  \]
\end{theorem}

\begin{remark}

Let $\Phi_k$ denote the $k$th cyclotomic polynomial.  We find that
\begin{equation}\label{cyclo}
(q)_{\lambda(N, t)}= \pm \prod_{k=1}^{\la (N,t)} \Phi_k (q)^{\lambda(N, tk)}.
\end{equation}
In particular, if $i\not\in S(t)$, then $A_t(t, i, q)$ is divisible by a high power of $\Phi_1(q)=q-1$.
This fact (when $t$ is prime) is crucial for  the proof of the congruences for Fishburn numbers which have been recently studied by 
Andrews-Sellers \cite{Andrews:2014},  Garvan \cite{Garvan:2014}, and Straub \cite{Straub:2014}. 
In the last section we illustrate how to derive congruences for the Fishburn numbers using Theorem~\ref{Atdiv}.
\end{remark}

\section{Proofs}

We require a lemma which shows that the quantities of interest are stable as $N$ grows.

\begin{lemma}\label{stablelemma}
Suppose that $\zeta$ is an $n$th root of unity, that $t\geq 1$ and that  $\nu\geq 0$.
\begin{enumerate}
\item If $N\geq (\nu+1) n-1$ then
\begin{equation}\label{Fstable}
\left( q \tfrac{d}{dq} \right)^{v}F(q)\big|_{q=\zeta}=\left( q \tfrac{d}{dq} \right)^{v}F(q; (\nu+1) n-1)\big|_{q=\zeta}=\left( q \tfrac{d}{dq} \right)^{v}F(q; N)\big|_{q=\zeta}.\end{equation}
\item If $N\geq (\nu+1) n t-1$ then for all $i$ we have
\[A^{(\nu )}_t(N, i, \zeta)=A^{(\nu )}_t\((\nu+1) n t-1, i, \zeta\).\]
\end{enumerate}
\end{lemma}

\begin{proof}
The first claim follows from the definition of $F(q; N)$.
For the second we require \cite[Lemma 2.4]{Andrews:2014} (which remains true for composite values of $t$).
In particular we have
\begin{equation}\label{diffrec}
\left( q \tfrac{d}{dq} \right)^\nu F(q; N) =\sum_{j=0}^{\nu}  \sum_{i=0}^{t-1} C_{\nu,i,j} (t) q^{i+jt} A_{t}^{(j)} (N, i, q^{t}),
\end{equation}
where $C_{\nu,i,j} (t)$ is the array of integers defined recursively as follows:
\begin{enumerate}
\item  $C_{0,0,0} (t) =1$,
\item $C_{\nu,i,0} (t) = i^\nu$ and $C_{\nu,i,j} (t) =0$ for $j \ge \nu+1$, 
\item $C_{\nu+1, i,j} (t) = (i + j t ) C_{\nu,i,j} (t) + t C_{\nu, i, j-1} (t)$ for $1\leq j\leq \nu$.
\end{enumerate}

We now extract those terms in 
 \eqref{diffrec} whose exponents are congruent to $i$ modulo $t$. By orthogonality, we find that
\begin{equation}\label{importanteq}
\sum_{j=0}^{\nu}  C_{\nu,i,j} (t) q^{i+jt} A_{t}^{(j)} (N, i , q^{t})=\frac{1}{t} \sum_{j=0}^{t-1} \zeta_{t}^{-ji}   \(  \left( q \tfrac{d}{dq} \right)^\nu F( q, N) \)\Big|_{q=\zeta_{t}^j q}.
\end{equation}
The second claim follows from induction on $\nu$ after substituting a $nt$-th root of unity for $q$,   using \eqref{Fstable} and the fact that the coefficients $C_{\nu, i, j}(t)$ which appear are non-zero.
\end{proof}
\begin{proof}[Proof of Theorem~\ref{Atdiv}]

Suppose that $t$ and $k$ are positive integers, that $i\not\in S(t)$, that    $\nu\geq 0$, and that $\zeta_k$ is a   $k$th root of unity. 
Using \eqref{cyclo},  the theorem will follow once we show that 
\begin{equation}\label{needtoshow}
A_t^{(\nu)}(N, i, \zeta_k)=0 \ \ \  \text{for}\ \ \  N\geq (\nu+1) t k -1.
\end{equation}
From   \eqref{atdef} we find  that 
\begin{equation}\label{basecase}
A_t (N, i, q ) = \sum_{j=0}^{k-1} q^{j} A_{tk} (N, i+jt, q^k).
\end{equation}

Note that if  $i \not\in S(t)$, then    $i+jt \not\in S(tk)$.
It is therefore enough to show (after replacing $A_{tk}$ by $A_t$) that for all $t$ and $\nu$, for $i\not\in S(t)$ and  for $N\geq (\nu+1) t  -1$ we have
\begin{equation}\label{needtoshow2}
A_{t}^{(\nu)}( N, i , 1) = 0.
\end{equation}

We prove \eqref{needtoshow2} by induction on $\nu$.
For the base case $\nu=0$, suppose that $N\geq t-1$.  
 Substituting $q=\zeta_{t}$  in \eqref{importanteq}, and using 
\eqref{Fstable}  and \eqref{Fexpansion} gives 
\[ A_{t}(N, i, 1)= \frac{1}{t}\sum_{j=1}^{t} \zeta_{t}^{-ji} F(\zeta_{t}^{j}; N)=\frac{1}{t}\sum_{j=1}^{t} \zeta_{t}^{-ji} F(\zeta_{t}^{j})=
\frac{1}{t}\sum_{j=1}^{t} \zeta_{t}^{-ji} b_0(\zeta_{t}^{j}).\]
We find that $A_{t}(N, i, 1)=0$ using \eqref{bndefine} and orthogonality, since $i\not\in S(t)$.

Suppose now that \eqref{needtoshow2} has been established for $j\leq \nu-1$ and that $N\geq (\nu+1) t  -1$.
Using \eqref{importanteq}  and \eqref{Fstable}  as in the base case along with the induction hypothesis, we find that
\[C_{\nu, i, \nu}(t)A_{t}^{(\nu)}(N, i, 1)=
\frac{1}{t} \sum_{j=1}^{t} \zeta_{t}^{-ji}   \(  \left( q \tfrac{d}{dq} \right)^\nu F( q) \)\Big|_{q=\zeta_{t}^j}
=\frac{1}{t}\sum_{j=1}^{t} \zeta_{t}^{-ji} b_\nu(\zeta_{t}^j).
\]
The result follows from \eqref{bndefine} and orthogonality (note that $C_{\nu, i, \nu}(t)$ is positive).
\end{proof}

\section{Application to congruences for Fishburn numbers}
As a corollary one can derive congruences for the Fishburn numbers modulo $p^r$ for any prime $p\geq 5$  (for background on these numbers, we refer the reader to \cite{Andrews:2014}, \cite{Garvan:2014} and \cite{Straub:2014}). For $r\geq 2$ these congruences and their generalizations to other types of Fishburn numbers
have recently been obtained by Straub \cite{Straub:2014}, who generalized the work of Andrews-Sellers \cite{Andrews:2014} and Garvan 
 \cite{Garvan:2014} in the case $r=1$.  We sketch a proof using Theorem~\ref{Atdiv} in the simplest case of the usual Fishburn numbers
$\xi(n)$ defined in the introduction.
Define $T(p^r)\subseteq\{0, 1, 2, \dots, p^r-1\}$ as the set of those residues which are strictly larger than any element 
of $S(p^r)$.  Using divisibility properties of binomial coefficients along with Theorem~\ref{Atdiv} we show that 
\begin{equation}\label{fishburncong}
j\in T(p^r)\implies \xi(p^r n+j)\equiv 0\pmod {p^r}\ \ \ \ \text{for all $n$}.
\end{equation}
\begin{remark}  It appears that this describes the complete set of such congruences  when $r\geq 2$.  When $r=1$, 
Garvan \cite{Garvan:2014} has proved a slightly stronger result for some primes $p$ (to the authors' knowledge,  $p=23$ is the only  prime
for which Garvan's result is known to be stronger for the usual Fishburn numbers).
For $r\geq 2$, \eqref{fishburncong} states that 
\[\xi(23^r n-j)\equiv 0\pmod {23^r}, \ \ \ 1\leq j\leq 5.\]
Garvan's work implies that the same is true when $r=1$.
\end{remark}
 
 To prove \eqref{fishburncong}, 
suppose that $n$ is a positive integer and that  $N\geq n p^r-1$.
By Theorem \ref{Atdiv} in the case when $t=p^r$,  there exists $P\in \Z[q]$ such that 
\begin{equation}\label{firststep}
F(1-q, N) = \sum_{ i \in S(p^r)} (1-q)^i A_{p^r} (N, i, (1-q)^{p^r}) + \(1-(1-q)^{p^r}\)^n\cdot P(q).
 \end{equation}

It follows that 
  \[ F(1-q, N) \equiv \sum_{ i \in S(p^r)} (1-q)^i A_{p^r} (N, i, (1-q)^{p^r}) +O \left(q^{pn-(p-1)(r-1)} \right)\pmod {p^r}.\]
The sum is an integral linear combination of terms of the form
\[(1-q)^{i+k p^r}\ \ \ \text{with} \ \ \ i\in S(p^r),\]
 so we must prove that if $m$ is a natural number whose residue modulo $p^r$ lies in $T(p^r)$, 
 then
 \begin{equation}\label{bi}
\binom{i + kp^r}{m} \equiv 0 \pmod{p^r}.
\end{equation}

Suppose that $n\in S(p)$ and that $r\geq 2$. These facts are straightforward to check:
\begin{enumerate}
\item If $n\not\equiv -1/24\pmod p$ then $n+pk\in S(p^r)$ for $0\leq k<p^{r-1}$.
\item If $n\equiv -1/24\pmod p$ then $n+p^r-p\not\in S(p^r)$.
\end{enumerate}
Letting  $n_0$ be the largest element of $S(p)$ which is not congruent to $-1/24\pmod p$, it follows that 
the largest element of $S(p^r)$ is $n_0+p^r-p$.
Since $m\in T(p^r)$, we have $m>n_0+p^r-p$.  Therefore  the $p$-adic expansion of $m$ has the form
\[m=m_0+(p-1)p+(p-1)p^2+\dots+(p-1)p^{r-1}+O(p^r),\]
where  $m_0$  is strictly larger than the the first digit in the $p$-adic expansion of $i$.
The congruence  \eqref{bi}  then follows by Kummer's theorem, and \eqref{fishburncong} follows from \eqref{firststep}
after letting $n\rightarrow\infty$.

\bibliographystyle{amsplain}
\bibliography{fishburn_bib.bib}

\providecommand{\bysame}{\leavevmode\hbox to3em{\hrulefill}\thinspace}
\providecommand{\MR}{\relax\ifhmode\unskip\space\fi MR }
\providecommand{\MRhref}[2]{%
  \href{http://www.ams.org/mathscinet-getitem?mr=#1}{#2}
}
\providecommand{\href}[2]{#2}
\begin{thebibliography}{1}

\bibitem{Andrews:2014}
George~E. Andrews and James~A. Sellers, \emph{Congruences for the {F}ishburn
  numbers}, \url{http://arxiv.org/abs/1401.5345} (2014).

\bibitem{Garvan:2014}
Frank Garvan, \emph{Congruences and relations for $r$-{F}ishburn numbers},
  Preprint (2014).

\bibitem{Straub:2014}
Armin Straub, \emph{Congruences for {F}ishburn numbers modulo prime powers},
  \url{http://arxiv.org/abs/1407.7521} (2014).

\bibitem{Zagier:2001}
Don Zagier, \emph{Vassiliev invariants and a strange identity related to the
  {D}edekind eta-function}, Topology \textbf{40} (2001), no.~5, 945--960.
  \MR{1860536 (2002g:11055)}

\bibitem{Zagier:2010}
\bysame, \emph{Quantum modular forms}, Quanta of maths, Clay Math. Proc.,
  vol.~11, Amer. Math. Soc., Providence, RI, 2010, pp.~659--675. \MR{2757599
  (2012a:11066)}

\end{thebibliography}

\end{document}